\documentclass{article}

\usepackage[a4paper, left=2.5cm, right=2.5cm, top=2.5cm, bottom=2.5cm]{geometry}
\usepackage{graphicx}
\usepackage{amsmath,amssymb}
\usepackage{amsthm}

\allowdisplaybreaks

\theoremstyle{plain}
\newtheorem{theorem}{Theorem}[section]
\newtheorem{proposition}[theorem]{Proposition}
\newtheorem{lemma}[theorem]{Lemma}
\newtheorem{corollary}[theorem]{Corollary}

\theoremstyle{definition}

\newtheorem{example}[theorem]{Example}

\theoremstyle{remark}
\newtheorem{remark}[theorem]{Remark}

\newcommand{\address}[1]{\par\small\textit{#1}\par}
\newcommand{\email}[1]{\par\small\texttt{#1}\par}

\title{Hypersurfaces with Non-diagonalizable Shape Operators in Minkowski Spaces}

\author{%
  Z\"ulal Derin Yaqub\thanks{Correspondence author, email: derinyaqub@itu.edu.tr}\\
  \address{Department of Mathematics, Istanbul Technical University, Istanbul 34000, T\"urkiye}
  \email{derinyaqub@itu.edu.tr}
  \and
  Nurettin Cenk Turgay\\
  \address{Department of Mathematics, Istanbul Technical University, Istanbul 34000, T\"urkiye}
  \email{turgayn@itu.edu.tr}
}

\date{}

\begin{document}

\maketitle

\begin{abstract}
	In this paper, we study hypersurfaces $M^4$ in the Minkowski space $\mathbb{E}_1^5$ whose shape operator $S$ is non-diagonalizable, having a triple real eigenvalue $\lambda$ with a one-dimensional eigenspace together with a simple eigenvalue $\mu$ (Case (iii) in the classification of Magid). Assuming that $\lambda$ and $\mu$ satisfy a functional relation $\lambda = f(\mu)$, we obtain a class of hypersurfaces which can be regarded as a common generalization of both Weingarten surfaces and rotational hypersurfaces. We first prove that $\lambda$ and $\mu$ are intrinsic invariants of $M$, and then establish the existence and uniqueness of isometric immersions of this type into $\mathbb{E}_1^5$.



\end{abstract}

\section{Introduction}

The study of surfaces and hypersurfaces in space forms whose principal curvature invariants satisfy a prescribed condition is one of the classical themes of differential geometry. In particular, a surface $M$ is said to be a \emph{Weingarten surface} if its principal curvatures $\lambda,\mu$ satisfy a functional relation of the form
$$W(\lambda,\mu)=0$$
for a function $W$. Equivalently, if $K$ and $H$ denote the Gaussian curvature and the mean curvature of $M$, respectively, this relation can be written as $U(K,H)=0$.

Such surfaces were first introduced by Weingarten \cite{Weingarten1861} in his study of surfaces that can be isometrically deformed onto one another. Since then they have been studied extensively in various ambient spaces. For example, Sa Earp and Toubiana \cite{SaEarpToubiana1999} established symmetry results, in the spirit of the Alexandrov reflection principle, for properly embedded special Weingarten surfaces in the hyperbolic space $\mathbb{H}^3$. López  obtained classification results for \emph{linear} Weingarten surfaces, that is, surfaces satisfying a relation of the form $aK+bH=c$, in both Euclidean and hyperbolic space, describing rotational examples, surfaces foliated by circles, and surfaces invariant under parabolic isometry groups in \cite{Lopez2008a,Lopez2008b}. More recently, Dursun \cite{Dursun2020} classified rotational Weingarten surfaces in $\mathbb{H}^3(-1)$ whose principal curvatures satisfy $\kappa = F(\lambda)$, determining the profile curves explicitly for several choices of $F$, including the linear and power-law cases.

On the other hand, another classical family of hypersurfaces that has attracted considerable attention is that of \emph{rotation} (or \emph{rotational}) \emph{hypersurfaces}. A hypersurface $M$ in a space form, or in a Riemannian product $S^n\times\mathbb{R}$ or $H^n\times\mathbb{R}$, is called a rotation hypersurface if it is invariant under a subgroup of isometries of the ambient space fixing a totally geodesic submanifold (the rotation axis), so that $M$ is generated by rotating a profile curve about this axis.

Do Carmo and Dajczer \cite{DoCarmoDajczer1983} showed that such a hypersurface has shape operator of the form $S = \mathrm{diag}(\lambda,\mu,\mu,\dots,\mu)$
where the principal curvature $\lambda$, associated with the unit tangent $e_1$ of the profile curve, satisfies a functional relation $\lambda = f(\mu)$, together with the condition $\nabla_{e_1}e_1=0$. Conversely, these two conditions -- namely, that the shape operator takes the above diagonal form with $\lambda=f(\mu)$, and that $\nabla_{e_1}e_1=0$ -- characterize rotation hypersurfaces in the space forms $S^n$, $H^n$, and $E^n$; this characterization was later extended to the product spaces $S^n\times\mathbb{R}$ and $H^n\times\mathbb{R}$ by Dillen, Fastenakels, and Van der Veken \cite{Dillen2009}.

The interplay between the rotational and Weingarten conditions has also been studied directly: Rosenberg and Sa Earp \cite{RosenbergSaEarp1994} investigated the geometry of properly embedded surfaces in $\mathbb{R}^3$ satisfying the special Weingarten relation $aH+bK=1$, many of which arise as rotation surfaces, while Yoon \cite{Yoon2005} classified Weingarten surfaces of revolution in the hyperbolic space $\mathbb{H}^3$.

In Lorentzian geometry, the indefinite metric allows the shape operator to fail to be diagonalizable. In this case, new geometric behavior does not have a counterpart in the Riemannian setting. The possible Jordan forms of such operators were classified by Magid \cite{Magid1985}, and hypersurfaces with non-diagonalizable shape operators have subsequently been investigated by several authors, see e.g.\ \cite{LiangXia2019}.

In this paper, we study hypersurfaces $M^4$ of the Minkowski space $\mathbb{E}_1^5$ whose shape operator $S$ has the non-diagonalizable form 
   \begin{equation} \label{eq:shape}
        S \sim
        \begin{pmatrix}
            \mu & 0 & 0 & 0 \\
            0 & \lambda & 0 & 1 \\
            0 & 0 & \lambda & 0 \\
            0 & 0 & -1 & \lambda
        \end{pmatrix}
    \end{equation}
for some smooth functions $\mu$ and $\lambda$ with respect to a suitable pseudo-orthonormal frame field (See Case (iii) of Lemma~\ref{lem:shape-classification}). Motivated by the classical notions discussed above, we impose a functional relation $\lambda = f(\mu)$ between the two principle curvatures of the hypersurfaces. We first show that, under this assumption, $\lambda$ and $\mu$ are intrinsic invariants of $M$, depending only on the induced metric. In particular, we establish the existence and uniqueness of isometric immersions of this type into $\mathbb{E}_1^5$. By this assumption, this class of hypersurfaces with non-diagonalizable shape operator may be viewed as a common generalization of Weingarten surfaces and rotational hypersurfaces in the Minkowski space.

The structure of the paper is organized as follows. In Section 2, we recall basic definitions and preliminaries. In Section 3, we derive curvature relations and obtain the fundamental equations. Section 4 is devoted to the main results, where we analyze the geometric structure and derive the governing system. Finally, we discuss existence and uniqueness of such hypersurfaces.

\section{Preliminaries}
	Let $\mathbb{E}_1^5$ denote the $5$-dimensional Minkowski space equipped with the metric tensor
$$	\bar{g} = \langle \cdot, \cdot \rangle = -dx_1^2  + \sum_{i=2}^5 dx_i^2,$$
	where $(x_1, x_2 \dots, x_5)$ is a Cartesian coordinate system with the Levi-Civita connection $\bar{\nabla}$. 
		
	\subsection{Basic Definitions and Facts on Hypersurfaces}	
		
		Consider a $4$-dimensional, oriented Lorentzian hypersurface $M$ immersed in $\mathbb{E}_1^5$ with the unit normal vector field $\xi$ and the Levi-Civita connections $\nabla$. Then, 	for any vector fields $X, Y$ the Gauss and Weingarten formul\ae, given by
\begin{eqnarray*}
	\bar{\nabla}_X Y &=& \nabla_X Y + h(X, Y),\\
	\bar{\nabla}_X N &=& -S X,
\end{eqnarray*}
define the second fundamental form $h$ and the shape operator $S$  of $M$. Note that $t$ and $h$  are related by $\langle h(X,Y), N \rangle = g(SX, Y)$, where $g = x^*\bar{g}$ denotes the induced metric on $M$. 

When $M$ is Lorentzian, the canonical algebraic classification for $S$ is well-known.
\begin{lemma}[\cite{Magid1985}]\label{lem:shape-classification}
	Let $M$ be a $4$-dimensional Lorentzian  manifold and $S$ be a symmetric endomorphism on $T_pM$. Then, $S$ takes one of the following canonical forms.
	
	\medskip
	\noindent
	\textbf{Case (i).} $S $ is diagonalizable.
	
	\medskip
	\noindent
	\textbf{Case (ii).} $\{e_1,e_2,e_3,e_4\}$ is a pseudo-orthonormal frame satisfying
	$$
	\langle e_1,e_2\rangle=-1, \qquad \langle e_3,e_3\rangle=\langle e_4,e_4\rangle=1,
	$$
	with all other products vanishing, and
	$$
	S \sim
	\begin{pmatrix}
		\lambda & 1 & 0 & 0 \\
		0& \lambda & 0 & 0 \\
		0 & 0 & \mu_1 & 0 \\
		0 & 0 & 0 & \mu_2
	\end{pmatrix}, \qquad \lambda,\mu_1,\mu_2\in\mathbb{R}.
	$$
	
	\medskip
	\noindent
	\textbf{Case (iii).} $\{e_1,e_2,e_3,e_4\}$ is a pseudo-orthonormal frame satisfying
	\begin{equation}\label{eq:shape0}
	\langle e_2, e_3 \rangle = -1, \qquad \langle e_1, e_1 \rangle = \langle e_4, e_4 \rangle = 1,
	\end{equation}
	with all other products vanishing, and $S$ satisfies \eqref{eq:shape}.

	\medskip
	\noindent
	\textbf{Case (iv).} $\{e_1,e_2,e_3,e_4\}$ is an orthonormal frame satisfying
	$$
	\langle e_1,e_1\rangle=-1, \qquad \langle e_2,e_2\rangle=\langle e_3,e_3\rangle=\langle e_4,e_4\rangle=1,
	$$
	and
	$$
	S \sim
	\begin{pmatrix}
\lambda & \nu & 0& 0 \\
-\nu & \lambda & 0& 0 \\
0 & 0 & \mu_1 & 0 \\
0 & 0 & 0& \mu_2 	\end{pmatrix}, \qquad \mu_1,\mu_2,\lambda,\nu \in\mathbb{R}, \ \nu \neq 0.$$
\end{lemma}

\begin{remark}\label{rem:assumptions}
\textbf{Assumptions.} Throughout this paper, we are going to assume that
\begin{enumerate}
    \item[(A1)] $M$ is a simply connected, smooth, $4$-dimensional Lorentzian hypersurface of $\mathbb{E}^5_1$ and the vector fields $e_1,e_2,e_3,e_4$ are smooth on $M$
    \item[(A2)] the shape operator $S$ of $M$ has the form given in case (iii) of Lemma~\ref{lem:shape-classification} for some smooth functions $\lambda,\mu$ such that $\lambda-\mu$ vanishes nowhere on $M$,
		\item[(A3)] $\lambda$ and $\mu$ satisfy
\begin{eqnarray}\label{lambndanufunctrelation01}
\lambda=f(\mu)
\end{eqnarray}
for a smooth function $f$.
    \item[(A4)] $e_1$ is a geodesic field, that is $\nabla_{e_1}e_1=0$.
\end{enumerate}
\end{remark}

We are going to use the following well-known theorem (See, for example, \cite{Dajczer1990} for Riemannian space-forms).
\begin{theorem}[Fundamental Theorem of Hypersurfaces \cite{ONeill1983}]\label{thm:FTH}
Let $N^{n+1}_{\nu}(c)$ denote the semi-Riemannian space form with constant sectional curvature $c$, let $\varepsilon\in\{-1,1\}$, let $(\Omega,g)$ be an $n$-dimensional simply connected semi-Riemannian manifold, and let $S$ be a field of $g$-symmetric linear operators on $\Omega$. Suppose that $S$ satisfies the Gauss equation
$$
R(X,Y)Z=\varepsilon\bigl(g(SY,Z)SX-g(SX,Z)SY\bigr)+c\bigl(g(Y,Z)X-g(X,Z)Y\bigr)
$$
and the Codazzi equation
$$
(\nabla_XS)Y=(\nabla_YS)X
$$
for all $X,Y,Z\in T\Omega$. Then,
\begin{itemize}
    \item \textbf{(Existence)} There exists an isometric immersion $\psi:(\Omega,g)\hookrightarrow N^{n+1}_{\nu}(c)$ such that $S$ is the shape operator of $\psi$ with respect to a unit normal field $\xi$ satisfying $\langle\xi,\xi\rangle=\varepsilon$.
    \item \textbf{(Uniqueness)} If $\psi,\tilde\psi:(\Omega,g)\hookrightarrow N^{n+1}_{\nu}(c)$ are two isometric immersions such that $S$ is the shape operator of both $\psi$ and $\tilde\psi$ with respect to unit normal fields $\xi$ and $\tilde\xi$ satisfying $\langle\xi,\xi\rangle=\langle\tilde\xi,\tilde\xi\rangle=\varepsilon$, then there exists an isometry $\tau:N^{n+1}_{\nu}(c)\to N^{n+1}_{\nu}(c)$ such that $\tilde\psi=\tau\circ\psi$.
\end{itemize}
\end{theorem}

\subsection{Sign conventions}
Let $M$ be a hypersurface of the Minkowski space $\mathbb E^5_1$. We denote  the curvature tensor of $M$ by $R$ given by
$$
R(X,Y)Z=\nabla_X\nabla_Y Z-\nabla_Y\nabla_X Z-\nabla_{[X,Y]}Z.
$$
The Ricci tensor is defined by 
\begin{equation}\label{eq:Ric}
\operatorname{Ric}(X)=\operatorname{tr} R(X,\cdot)\cdot
\end{equation}
and scalar curvature $\operatorname{Scal}$ of $M$ is the trace of $\operatorname{Ric}$, that is,
\begin{equation}\label{eq:Scal}
\operatorname{Scal}=\operatorname{tr}\operatorname{Ric}(\cdot,\cdot),
\end{equation}
where we put
$$\operatorname{Ric}(X,Y)=R(\cdot,X,Y,\cdot).$$

Moreover, the Gauss and Codazzi equations are
\begin{eqnarray}	
\label{eq:gauss}  R(X,Y,Z,W) &=& \langle h(Y,Z), h(X,W) \rangle - \langle h(X,Z), h(Y,W) \rangle,\\
\label{eq:codazzi} (\nabla_X h)(Y,Z) &=& (\nabla_Y h)(X,Z)
\end{eqnarray}
satisfied whenever $X, Y, Z, W $ are tangent to $M$, where the  covariant derivative $(\nabla_X h)$ of $h$ is defined by
$$	(\bar{\nabla}_X h)(Y, Z) = \nabla_X^\perp h(Y, Z) - h(\nabla_X Y, Z) - h(Y, \nabla_X Z).$$

\section{Hypersurfaces with Non-Diagonalizable Shape Operator}

In this section, we study some geometrical properties of hypersurfaces
in $\mathbb{E}_1^5$ with the shape operator given in case (iii) of Lemma \ref{lem:shape-classification}.

Let $M$ be a hypersurface satisfying the assumptions give in Remark \ref{rem:assumptions}. Then, the second fundamental form $h$ of $M$ takes the form
\begin{equation}\label{eq:second-fundamental-form}
	h(e_1,e_1) = \mu,  \quad	h(e_2,e_3) = -\lambda,  \quad	h(e_3,e_4) = -1,\quad 	h(e_4,e_4) = \lambda,
\end{equation}
and $h(e_i,e_j)=0$ for all other pairs $(i,j)$.

Since $\{e_1,e_2,e_3,e_4\}$ satisfies \eqref{eq:shape0}, by using connection $1$-forms $\omega_{ij}$ defined by
$$\omega_{ij}(e_k) = \langle \nabla_{e_k} e_i, e_j \rangle,$$
which are skew-symmetric, i.e., $\omega_{ij} = -\omega_{ji}$, one can observe that $\nabla$ takes the form
\begin{align}\label{eq:nabla11}
\begin{split}
	\nabla_{e_i} e_1 &= -\omega_{13}(e_i)  e_2 - \omega_{12}(e_i)  e_3 + \omega_{14}(e_i)  e_4, \\
	\nabla_{e_i} e_2 &= -\omega_{12}(e_i)  e_1 - \omega_{23}(e_i)  e_2 + \omega_{24}(e_i)  e_4, \\
	\nabla_{e_i} e_3 &= -\omega_{13}(e_i)  e_1 + \omega_{23}(e_i)  e_3 + \omega_{34}(e_i)  e_4, \\
	\nabla_{e_i} e_4 &= -\omega_{14}(e_i)  e_1 + \omega_{34}(e_i)  e_2 + \omega_{24}(e_i)  e_3.
\end{split}
\end{align}

First, we obtain the following lemma.
\begin{lemma}\label{lem:ConnForms}
Let $M$ be a hypersurface satisfying the assumptions given in Remark~\ref{rem:assumptions}. Then the connection forms $\omega_{ij}$ and the principal curvatures $\lambda,\mu$ of $M$ defined by \eqref{eq:nabla11} satisfy
\begin{subequations}\label{eq:All:lem:ConnForms}
\begin{align}
& e_a(\mu) = e_a(\lambda) = 0, \qquad a=2,3,4, \label{eq:muconst}\\
& \omega_{12}(e_1)=\omega_{13}(e_1)=\omega_{14}(e_1)=0, \label{eq:wzero1}\\
& \omega_{12}(e_2)=\omega_{14}(e_2)=\omega_{23}(e_2)=\omega_{24}(e_2)=\omega_{12}(e_4)=\omega_{24}(e_4)=0, \label{eq:wzero2}\\
& \omega_{23}(e_3)=2\omega_{34}(e_4),\quad \omega_{23}(e_4)=4\omega_{34}(e_2),\quad \omega_{24}(e_3)=2\omega_{34}(e_2), \label{eq:wrel1}\\
& \omega_{14}(e_3)=(\lambda-\mu)\omega_{13}(e_3)+2\omega_{34}(e_1),\qquad
  \omega_{14}(e_4)=(\lambda-\mu)\omega_{13}(e_4)+\omega_{23}(e_1), \label{eq:wrel2}\\
& \omega_{12}(e_3)=\omega_{13}(e_2)=-\omega_{23}(e_1)-(\lambda-\mu)\big((\lambda-\mu)\omega_{13}(e_3)+2\omega_{34}(e_1)\big), \label{eq:wrel3}\\
& \omega_{24}(e_1)=\tfrac13(\lambda-\mu)^2\big((\lambda-\mu)\omega_{13}(e_3)-\omega_{13}(e_4)+2\omega_{34}(e_1)\big), \label{eq:wrel4}\\
& e_1(\lambda)=-(\lambda-\mu) \omega_{23}(e_1)-\tfrac13(\lambda-\mu)^2\big(2(\lambda-\mu) \omega_{13}(e_3)+\omega_{13}(e_4)+4 \omega_{34}(e_1)\big). \label{eq:e1lambda}
\end{align}
\end{subequations}
\end{lemma}

\begin{proof}
The assumption (A4) implies \eqref{eq:wzero1}. We apply the Codazzi equation \eqref{eq:codazzi} with $X=e_i$, $Y=e_j$ and $Z=e_k$ for the triplets $(i,j,k)$ listed below.

The triplets $(1,a,1)$, $a=2,3,4$, together with \eqref{lambndanufunctrelation01} and \eqref{eq:wzero1}, give \eqref{eq:muconst}. From the triplets $(1,2,2)$, $(1,2,4)$, $(1,4,2)$, $(2,3,2)$, $(2,3,4)$ and $(3,4,2)$ we obtain \eqref{eq:wzero2}. Next, the triplets $(3,4,3)$, $(2,3,3)$ and $(3,4,4)$ give \eqref{eq:wrel1}, and the triplets $(1,3,3)$ and $(1,4,3)$ give \eqref{eq:wrel2}.

Moreover, the triplet $(2,3,1)$ gives $\omega_{12}(e_3)=\omega_{13}(e_2)$, while the triplet $(1,3,4)$ gives $\omega_{12}(e_3)=-\omega_{23}(e_1)-(\lambda-\mu) \omega_{14}(e_3)$. Combining these with the first equation in \eqref{eq:wrel2} yields \eqref{eq:wrel3}.

Finally, the triplets $(1,3,2)$ and $(1,4,4)$ give
\begin{equation*}
e_1(\lambda)=(\lambda-\mu) \omega_{12}(e_3)+\omega_{24}(e_1)
\quad\text{and}\quad
e_1(\lambda)=-(\lambda-\mu) \omega_{14}(e_4)-2 \omega_{24}(e_1),
\end{equation*}
respectively. By eliminating $e_1(\lambda)$ from these equations and using \eqref{eq:wrel2} and \eqref{eq:wrel3}, we obtain \eqref{eq:wrel4}. Then \eqref{eq:e1lambda} follows from the second of these equations. The remaining triplets are identically satisfied by virtue of \eqref{eq:All:lem:ConnForms}.
\end{proof}

As a corollary of Lemma \ref{lem:ConnForms}, we have 
\begin{corollary}\label{cor:lie-brackets}
Let $M$ be a hypersurface satisfying the assumptions give in Remark \ref{rem:assumptions}. Then, the vector fields $e_1,e_2,e_3,e_4$ satisfy
	\begin{subequations}\label{eq:cor:lie-brackets}
	\begin{align}
		[e_1,e_2] &= \left[(\mu-\lambda)\omega_{13}(e_3)-3\omega_{23}(e_1)\right]e_2 + (\mu-\lambda) \omega_{14}(e_3) e_4, \\
		[e_1,e_3] &= \omega_{13}(e_3) e_2 + \left[(\mu-\lambda)\omega_{13}(e_3)-\omega_{23}(e_1)\right]e_3 + \left[\omega_{34}(e_1)-\omega_{14}(e_3)\right]e_4, \\
		[e_1,e_4] &= \left[\omega_{14}(e_3)+\omega_{34}(e_1)\right]e_2 + (\mu-\lambda) \omega_{14}(e_3) e_3 + \left[(\mu-\lambda)\omega_{13}(e_3)-2\omega_{23}(e_1)\right]e_4, \\
		[e_2,e_3] &= \omega_{23}(e_3) e_2 + \left(\omega_{34}(e_2)-2\omega_{23}(e_4)\right)e_4, \\
		[e_2,e_4] &= \left(\omega_{34}(e_2)+\omega_{23}(e_4)\right)e_2, \\
		[e_3,e_4] &= \omega_{34}(e_3) e_2 + \omega_{23}(e_4) e_3 - \omega_{34}(e_4) e_4.
	\end{align}
	\end{subequations}
\end{corollary}

\begin{proof}
The proof follows from \eqref{eq:nabla11} and \eqref{eq:All:lem:ConnForms}.
\end{proof}

%%%%%%%%%%%%%%%%%%%%%%
%%%%%%%%%%%%%%%%%%%%%%
%%%%%%%%%%%%%%%%%%%%%%
%%%%%%%%%%%%%%%%%%%%%%
%%%%%%%%%%%%%%%%%%%%%%

\section{Uniqueness of Hypersurfaces}
In this section, we prove that hypersurfaces in $\mathbb E^5_1$ with the shape operator given in case (iii) of Lemma \ref{lem:shape-classification} exist and are unique up to isometries of $\mathbb E^5_1$.

\begin{lemma}\label{lem:RicciOperator}
Let $M$ be a hypersurface satisfying the assumptions given in Remark~\ref{rem:assumptions} and $p\in M$. Then, we have the following:
\begin{enumerate}
    \item[(i)] If $\lambda(p)+\mu(p)\neq0$, then $\operatorname{Ric}_p$ has the form given in case (iii) of Lemma~\ref{lem:shape-classification};
    \item[(ii)] If $\lambda(p)+\mu(p)=0$, then $\operatorname{Ric}_p$ has the form given in case (ii) of Lemma~\ref{lem:shape-classification};
    \item[(iii)] The functions
		\begin{equation}\label{eq:lem:RicciOperator:eigenvalues}
    r_1=3\lambda\mu, \qquad r_2=2\lambda^2+\lambda\mu
		\end{equation}
		are the only eigenvalues of $\operatorname{Ric}$ such that  $\operatorname{Ric}e_1=r_1e_1$, $\operatorname{Ric}e_2=r_2e_2$.
		\item[(iv)] The scalar curvature of $M$ is 
		\begin{equation}\label{eq:lem:RicciOperator:scal}
    \operatorname{Scal}=6\lambda^2+6\lambda\mu
		\end{equation}
		and it satisfies $e_a(\operatorname{Scal})=0, \ a=2,3,4$.
\end{enumerate}
\end{lemma}

\begin{proof} 
By using the Gauss equation \eqref{eq:gauss} and \eqref{eq:second-fundamental-form}, we obtain that the Ricci operator $\operatorname{Ric}$ of $M$ is given by
\begin{equation}\label{eq:RicciOperator}
\operatorname{Ric}=\left(
\begin{array}{cccc}
 3\lambda\mu & 0 & 0 & 0 \\
 0 & 2\lambda^2+\lambda\mu & 1 & \lambda+\mu \\
 0 & 0 & 2\lambda^2+\lambda\mu & 0 \\
 0 & 0 & -\lambda-\mu & 2\lambda^2+\lambda\mu
\end{array}
\right)
\end{equation}
with respect to the frame field $\{e_1,e_2,e_3,e_4\}$. Proof of (i)-(iv) directly follows from \eqref{eq:RicciOperator}.
\end{proof}

\begin{lemma}\label{lem:intrinsic}
Let $\psi:(\Omega,g)\hookrightarrow\mathbb{E}^5_1$ be an isometric immersion of a $4$-dimensional Lorentzian manifold $(\Omega,g)$ such that $M:=\psi(\Omega)$ satisfies the assumptions given in Remark~\ref{rem:assumptions}. Then the vector fields $e_1,e_2,e_3,e_4$ and the functions $\lambda,\mu$ are intrinsic. More precisely, $e_1,e_2,e_3,e_4$ are uniquely determined by the metric $g$ up to their directions, and $\lambda$ and $\mu$ are intrinsic up to a simultaneous change of sign.
\end{lemma}

\begin{proof}
\textit{Step 1. Identifying the points at which $\lambda=0$ and $\mu=0$.} Let $p\in M$. By a direct computation using \eqref{eq:RicciOperator}, one can observe that
\begin{align*}
\lambda(p)=0 &\iff \dim\ker\operatorname{Ric}_p=2,\\
\mu(p)=0 &\iff \dim\ker\operatorname{Ric}_p=1 \text{ and } \ker\operatorname{Ric}_p \text{ is non-degenerate}.
\end{align*}

\textit{Step 2. Identifying $e_1$.} Because of Lemma~\ref{lem:ConnForms}, we have $e_a(\lambda)=e_a(\mu)=0$ for $a=2,3,4$. Therefore, for any smooth function $f=f(\lambda,\mu)$ we have
\begin{equation}\label{eq:Step2:grad}
\nabla f=e_1(f) e_1 .
\end{equation}
In particular, from \eqref{eq:lem:RicciOperator:eigenvalues} we obtain
\begin{equation}\label{eq:Step2:e1r}
e_1(r_1)=3\bigl(\mu  e_1(\lambda)+\lambda  e_1(\mu)\bigr),\qquad
e_1(r_2)=(4\lambda+\mu)  e_1(\lambda)+\lambda  e_1(\mu).
\end{equation}
Consider the sets
$$\mathcal U_1=\{p\in M:\ \nabla r_1(p)\neq0 \text{ or } \nabla r_2(p)\neq0\},\qquad \mathcal U_2=\{p\in M:\ \lambda(p)\mu(p)=0\}, $$
and $\mathcal A=M\setminus(\mathcal U_1\cup\mathcal U_2)$. Since $r_1,r_2$ are eigenvalues of $\operatorname{Ric}$, $\mathcal U_1$ is intrinsic. Because of Step~1, $\mathcal U_2$ is also intrinsic.

\textit{Case 1. $p\in\mathcal U_1$.} Because of \eqref{eq:Step2:grad}, we have
$$e_1(p)=\pm\frac{\nabla r_i}{|\nabla r_i|}(p),$$
where $i\in\{1,2\}$ is chosen such that $\nabla r_i(p)\neq0$.

\textit{Case 2. $p\in\mathcal U_2$.} First, assume $\lambda(p)=0$. Then $\mu(p)\neq0$ and $\ker\operatorname{Ric}_p=\operatorname{span}\{e_1,e_2\}$. Therefore, a unit space-like vector $X\in\ker\operatorname{Ric}_p$ has the form $X=\pm(e_1+a e_2)$ for some $a\in\mathbb R$. Put
$$W_a:=X^\perp=\operatorname{span}\{e_2,\ a e_1+e_3,\ e_4\}$$
and let $R(W_a,W_a)W_a$ denote the subspace spanned by $\{R(Y,Z)U:\ Y,Z,U\in W_a\}$. By a direct computation using the Gauss equation \eqref{eq:gauss}, we obtain
$$R(W_a,W_a)W_a=\operatorname{span}\bigl\{-a\mu(e_1+a e_2)+e_4,\ e_2\bigr\}.$$
On the other hand, $(\ker\operatorname{Ric}_p)^\perp=\operatorname{span}\{e_2,e_4\}$. Since $\mu(p)\neq0$, we get
$$R(W_a,W_a)W_a=(\ker\operatorname{Ric}_p)^\perp \iff a=0 \iff X=\pm e_1(p).$$
Hence $e_1(p)$ is, up to its direction, the unique unit space-like vector $X\in\ker\operatorname{Ric}_p$ satisfying $R(X^\perp,X^\perp)X^\perp=(\ker\operatorname{Ric}_p)^\perp$.

Next, assume $\mu(p)=0$. Then $\ker\operatorname{Ric}_p=\operatorname{span}\{e_1\}$, and $e_1(p)$ is, up to sign, the unit vector in $\ker\operatorname{Ric}_p$.

\textit{Case 3. $p\in\mathcal A$.} We claim that $\mathcal A$ has empty interior. Assume, on the contrary, that $\mathcal A$ contains a non-empty connected open subset $\mathcal O$. On $\mathcal O$ we have $\lambda\mu\neq0$ and $e_1(r_1)=e_1(r_2)=0$. From \eqref{eq:Step2:e1r}, we get $3e_1(r_2)-e_1(r_1)=12\lambda  e_1(\lambda)$, which gives $e_1(\lambda)=0$ and then $e_1(\mu)=0$. Combining this with $e_a(\lambda)=e_a(\mu)=0$, $a=2,3,4$, we see that $\lambda$ and $\mu$ are constant on $\mathcal O$, i.e., $\mathcal O$ is an isoparametric hypersurface. Because of \cite[Corollary~2.7]{Magid1985} and (A2), we get $\lambda\mu=0$ on $\mathcal O$, which is a contradiction.

Therefore, $M\setminus\mathcal A$ is dense in $M$. Consider the $(1,1)$-tensor field $P(X)=\langle X,e_1\rangle e_1$. Because of Case~1 and Case~2, $P$ is determined by $g$ on $M\setminus\mathcal A$. Since $P$ is continuous, it is determined by $g$ on $M$. Since $P_p$ determines $e_1(p)$ up to sign, $e_1$ is intrinsic.

\textit{Step 3. Identifying $\lambda,\mu$.} Because of Step~2, $e_1$ is intrinsic. Therefore, the functions
$$r_1=\operatorname{Ric}(e_1,e_1),\qquad r_2=\tfrac13\bigl(\operatorname{Scal}-r_1\bigr)$$
are intrinsic. Because of Lemma~\ref{lem:RicciOperator}, they are the eigenvalues given in \eqref{eq:lem:RicciOperator:eigenvalues}, i.e., $\operatorname{Ric}(e_1)=r_1e_1$ and $\operatorname{Ric}(e_2)=r_2e_2$.

\textit{Case 1. $\lambda(p)\neq0$.} From \eqref{eq:lem:RicciOperator:eigenvalues}, we have $3r_2-r_1=6\lambda^2$ and $r_1=3\lambda\mu$. Therefore, $3r_2-r_1>0$ at $p$ and
\begin{equation}\label{eq:Step3:lambdamu}
\lambda=\varepsilon\sqrt{\frac{3r_2-r_1}{6}},\qquad
\mu=\varepsilon \frac{\sqrt2 r_1}{\sqrt{3(3r_2-r_1)}},\qquad \varepsilon\in\{-1,1\}.
\end{equation}

\textit{Case 2. $\lambda(p)=0$.} In this case $r_1=r_2=0$, so \eqref{eq:Step3:lambdamu} gives no information about $\mu$. Consider the operators $K,D$ on $e_1(p)^\perp$ defined by
$$K(X)=R(X,e_1)e_1,\qquad D=\operatorname{Ric}-K .$$
Because of Step~2, $K$ and $D$ are intrinsic. By a direct computation using the Gauss equation \eqref{eq:gauss} and \eqref{eq:RicciOperator}, we obtain
$$  K(e_2)=0,\quad K(e_3)=-\mu e_4,\quad K(e_4)=\mu e_2,\qquad D(e_2)=D(e_4)=0,\quad D(e_3)=e_2 .$$
Therefore, $D\neq0$ and $K^2=-\mu^2D$. Hence $\mu^2$ is intrinsic, and $\mu=\varepsilon\sqrt{\mu^2}$ with $\varepsilon\in\{-1,1\}$.

Finally, let $(\tilde\lambda,\tilde\mu)$ be the principal curvatures corresponding to another isometric immersion of $(\Omega,g)$ satisfying the assumptions given in Remark~\ref{rem:assumptions}. The above argument shows that $(\tilde\lambda,\tilde\mu)=\varepsilon(\lambda,\mu)$ for a function $\varepsilon:M\to\{-1,1\}$. Since $\lambda-\mu$ vanishes nowhere, $\varepsilon=(\tilde\lambda-\tilde\mu)/(\lambda-\mu)$ is continuous. Since $M$ is connected, $\varepsilon$ is constant. Hence $\lambda$ and $\mu$ are intrinsic up to a simultaneous change of sign.

\textit{Step 4. Identifying $e_2,e_3,e_4$.} Because of Step~2, $e_1$ is intrinsic, and so is the $\operatorname{Ric}$-invariant subspace $e_1(p)^\perp=\operatorname{span}\{e_2,e_3,e_4\}$. Because of Lemma~\ref{lem:RicciOperator}, the sets $\{p\in M:\ \lambda(p)+\mu(p)\neq0\}$ and $\{p\in M:\ \lambda(p)+\mu(p)=0\}$ are intrinsic. We consider these two cases separately.

\textit{Case 1. $\lambda(p)+\mu(p)\neq0$.} Because of Lemma~\ref{lem:RicciOperator}(i), there exists a pseudo-orthonormal frame $\{f_2,f_3,f_4\}$ of $e_1(p)^\perp$, i.e.,
$$\langle f_2,f_3\rangle=-1,\qquad \langle f_4,f_4\rangle=1,\qquad \langle f_i,f_j\rangle=0 \text{ otherwise},$$
such that
\begin{equation}\label{eq:Step3:Case1}
\operatorname{Ric}(f_2)=r_2f_2,\qquad
\operatorname{Ric}(f_3)=r_2f_3-f_4,\qquad
\operatorname{Ric}(f_4)=r_2f_4+f_2 .
\end{equation}
Such a frame is unique up to sign. Indeed, if $\{\tilde f_2,\tilde f_3,\tilde f_4\}$ is another such frame, the linear map $f_i\mapsto\tilde f_i$ is an isometry of $e_1(p)^\perp$ commuting with $N=\operatorname{Ric}-r_2 \mathrm{Id}$. Hence it has the form $a \mathrm{Id}+bN+cN^2$, and the conditions on the inner products give $b=c=0$, $a=\pm1$. By a direct computation using \eqref{eq:RicciOperator}, we obtain
\begin{equation}\label{eq:Step3:Case1:e}
e_2=-\frac{f_2}{\lambda+\mu},\qquad
e_3=-\frac{f_2}{8(\lambda+\mu)^3}-(\lambda+\mu)f_3-\frac{f_4}{2(\lambda+\mu)},\qquad
e_4=-\frac{f_2}{2(\lambda+\mu)^2}-f_4 .
\end{equation}

\textit{Case 2. $\lambda(p)+\mu(p)=0$.} In this case $r_2=\mu^2$ and $\mu(p)\neq0$. Because of Lemma~\ref{lem:RicciOperator}(ii), there exists a pseudo-orthonormal frame $\{f_2,f_3,f_4\}$ of $e_1(p)^\perp$ such that
\begin{equation}\label{eq:Step3:Case2}
\operatorname{Ric}(f_2)=\mu^2f_2,\qquad
\operatorname{Ric}(f_3)=\mu^2f_3+f_2,\qquad
\operatorname{Ric}(f_4)=\mu^2f_4 .
\end{equation}
In contrast to Case~1, such a frame is not unique. By a direct computation, the frames satisfying \eqref{eq:Step3:Case2} are exactly
\begin{eqnarray*}
f_2&=&\varepsilon e_2,\\
f_3&=&\varepsilon\Bigl(e_3+\frac{\beta^2}{2}e_2\Bigr)+\beta e_4,\\
f_4&=&\delta\bigl(e_4+\beta e_2\bigr),
\end{eqnarray*}
where $\varepsilon,\delta\in\{-1,1\}$ and $\beta\in\mathbb R$. On the other hand, the Gauss equation \eqref{eq:gauss} gives
$$R(e_1,f_3,e_1,f_3)=2\varepsilon\mu\beta .$$
Since $\mu(p)\neq0$, the additional condition $R(e_1,f_3,e_1,f_3)=0$ holds if and only if $\beta=0$. Therefore, the frames satisfying \eqref{eq:Step3:Case2} and $R(e_1,f_3,e_1,f_3)=0$ are exactly
$$f_2=\varepsilon e_2,\qquad f_3=\varepsilon e_3,\qquad f_4=\delta e_4,\qquad \varepsilon,\delta\in\{-1,1\}.$$
Since $R(e_1,\cdot ,e_1,\cdot)$ is intrinsic because of Step~2, $e_2,e_3,e_4$ are determined up to sign at $p$.
\end{proof}

Now, we are ready to prove the main result of this section:
\begin{theorem}\label{thm:uniqueness}
Let $\psi,\tilde\psi:(\Omega,g)\hookrightarrow\mathbb{E}^5_1$ be two isometric immersions of a $4$-dimensional Lorentzian manifold $(\Omega,g)$ such that $\psi(\Omega)$ and $\tilde\psi(\Omega)$ satisfy the assumptions given in Remark~\ref{rem:assumptions}. Then, there exists an isometry $\tau:\mathbb{E}^5_1\to\mathbb{E}^5_1$ such that $\tilde\psi=\tau\circ\psi$.
\end{theorem}

\begin{proof}
Let $S$ and $\tilde S$ be the shape operators of $\psi$ and $\tilde\psi$ with respect to the unit normal fields $\xi$ and $\tilde\xi$, and let $\{e_1,e_2,e_3,e_4\}$, $\lambda,\mu$ and $\{\tilde e_1,\tilde e_2,\tilde e_3,\tilde e_4\}$, $\tilde\lambda,\tilde\mu$ be the corresponding pseudo-orthonormal frame fields and principal curvatures. Because of Lemma~\ref{lem:intrinsic}, we have $\tilde e_1=\pm e_1$ and $(\tilde\lambda,\tilde\mu)=\varepsilon(\lambda,\mu)$ for a constant $\varepsilon\in\{-1,1\}$. We are going to show that $\tilde S=\varepsilon S$ on $\Omega$.

Let $P(X)=\langle X,e_1\rangle e_1$ and $X^\perp=X-P(X)$. We also consider the operators
$$K(X)=R(X,e_1)e_1,\qquad N=\operatorname{Ric}-r_2 \mathrm{Id}$$
on $e_1^\perp$. Since $\tilde e_1=\pm e_1$ and $\tilde r_2=r_2$, these operators are the same for $\psi$ and $\tilde\psi$. Because of the Gauss equation \eqref{eq:gauss} and \eqref{eq:RicciOperator}, a direct computation gives
\begin{align}
S(X)&=\mu P(X)+\frac1\mu K(X^\perp) && \text{if } \mu\neq0, \label{eq:thm:uniqueness:S1}\\
S(X)&=\mu P(X)+\lambda X^\perp+\frac{1}{\lambda+\mu}N(X^\perp)+\frac{1}{(\lambda+\mu)^3}N^2(X^\perp) && \text{if } \lambda+\mu\neq0. \label{eq:thm:uniqueness:S2}
\end{align}
Since $\lambda-\mu$ vanishes nowhere, at each point of $\Omega$ at least one of $\mu$ and $\lambda+\mu$ is non-zero. The right-hand sides of \eqref{eq:thm:uniqueness:S1} and \eqref{eq:thm:uniqueness:S2} are odd in $(\lambda,\mu)$. Therefore $\tilde S=\varepsilon S$ on $\Omega$.

Now, by replacing $\tilde\xi$ with $\varepsilon\tilde\xi$ if necessary, we let $\tilde S=S$. Hence, the uniqueness part of Theorem~\ref{thm:FTH} implies the existence of an isometry $\tau$ of $\mathbb E^5_1$ such that $\tilde\psi=\tau\circ\psi$.
\end{proof}

%%%%%%%%%%%%%%%%%%%%%%
%%%%%%%%%%%%%%%%%%%%%%
%%%%%%%%%%%%%%%%%%%%%%
%%%%%%%%%%%%%%%%%%%%%%
%%%%%%%%%%%%%%%%%%%%%%

\section{Warped Products with Non-Diagonalizable Shape Operator}
In this section, we are going to obtain a classification of warped product hypersurfaces whose shape operator has the form given in  case (iii) of Lemma~\ref{lem:shape-classification}.

First, we construct an example of hypersurface with non-diagonalizable shape operator.
\begin{example}\label{ex:Warped}
Let $\alpha=(\phi,\varphi)\colon I\to\mathbb{R}^2$ be a smooth arc-length parametrized curve and $N^3_1$ be a Lorentzian manifold.  For a given $\lambda_0\in \mathbb R$, let $\hat\psi:N\to \mathbb S^4_1$ be an isometric immersion such that $\hat\psi(N)$ is a Lorentzian isoparametric hypersurface of $\mathbb S^4_1$ whose shape operator has the matrix representation \eqref{eq:shape:hatM} with respect to the pseudo-orthonormal frame field $\{E_1,E_2,E_3\}$. Consider the mapping $\psi:(I\times N,g)\to\mathbb E^5_1$ given by 
\begin{align}\label{ex:Warped:PosVect}
\psi(x,(y,u,v))&=\varphi(x)(i\circ\hat\psi)(y,u,v)+\phi(x) (i\circ\hat \xi)(y,u,v),
\end{align}
where $\hat \xi$ is the unit normal vector field of $\hat\phi$, $i:\mathbb S^4_1\subset \mathbb E^5_1$ is the inclusion.
\end{example}

\begin{proposition}\label{prop:warpedproduct1:ex:Warped}
Let $\psi$ be  the mapping $\psi$ defined in Example \ref{ex:Warped} for some $\alpha=(\phi,\varphi)$ and $\lambda_0$.  Then, we have the followings:
\begin{enumerate}
\item[(i)] $\psi$ is an isometric immersion if and only if $\phi,\varphi$ and $\lambda_0$ satisfy
\begin{equation}\label{eq1:prop:warpedproduct1:ex:Warped}
\varphi (x)-\lambda _0 \phi (x)\neq 0
\end{equation}
for all $x\in I$.
\item[(ii)] When \eqref{eq1:prop:warpedproduct1:ex:Warped} is satisfied, $g$ is warped product if and only if $\alpha$ lies on a line through the origin.

\item[(iii)] When \eqref{eq1:prop:warpedproduct1:ex:Warped} is satisfied, the shape operator $S$ of $\psi$ has the matrix representation given in Case (iii) of Lemma~\ref{lem:shape-classification} if and only if $\alpha$  satisfies
\begin{equation}\label{eq2:prop:warpedproduct1:ex:Warped}
\varphi (x) \varphi '(x)+\phi (x) \phi '(x)\neq 0
\end{equation}
for all $x\in I$.

\item[(iv)] When \eqref{eq1:prop:warpedproduct1:ex:Warped} and \eqref{eq2:prop:warpedproduct1:ex:Warped} are satisfied, $S$  satisfies \eqref{eq:shape} for the functions
\begin{equation*}
\mu=\kappa, \qquad \lambda(x)=\frac{\lambda _0 \varphi '(x)+\phi '(x)}{\lambda _0 \phi (x)-\varphi (x)}
\end{equation*}
and the vector fields
\begin{align*}
e_1=&\partial_x,  \qquad e_2=\frac{\varphi   \varphi ' +\phi   \phi ' }{\left(\varphi  -\lambda _0 \phi  \right){}^3} E_1^h,\\
e_3=&\frac{3 \phi  ^2}{8 \left(\lambda _0 \phi  -\varphi  \right) \left(\varphi   \varphi ' +\phi   \phi ' \right)}E_1^h+\frac{\varphi  -\lambda _0 \phi  }{\varphi   \varphi ' +\phi   \phi ' }E_2^h\\&-\frac{\phi  }{2 \varphi   \varphi ' +2 \phi   \phi ' }E_3^h,\\
e_4=&-\frac{3 \phi  }{2 \left(\varphi  -\lambda _0 \phi  \right){}^2}E_1^h+\frac{1}{\lambda _0 \phi  -\varphi  }E_3^h,
\end{align*}
where $\kappa$ is the curvature of $\alpha$.
\end{enumerate}
\end{proposition}

\begin{proof}
By a direct computation using \eqref{ex:Warped:PosVect}, we obtain 
$$d\psi=\varphi\mathrm{\,id} -\phi\hat S$$
whenever $X$ is tangent to $N$. Therefore, the metric $g$ has the matrix representation
\begin{equation}\label{eq3:prop:warpedproduct1:ex:Warped}
g\sim\left(
\begin{array}{cccc}
 1 & 0 & 0 & 0 \\
 0 & 0 & -\left(\varphi  -\lambda _0 \phi  \right){}^2 & 0 \\
 0 & -\left(\varphi  -\lambda _0 \phi  \right){}^2 & \phi  ^2 & 2 \phi   \left(\varphi  -\lambda _0 \phi  \right) \\
 0 & 0 & 2 \phi   \left(\varphi  -\lambda _0 \phi  \right) & \left(\varphi  -\lambda _0 \phi  \right){}^2 \\
\end{array}
\right)
\end{equation}
with respect to $\{\partial_x,E_1^h,E_2^h,E_3^h\}$. So, we have (i) of the proposition. By a further computation using \eqref{eq3:prop:warpedproduct1:ex:Warped}, we get (ii) of the proposition.

When the mapping $\psi$ is an isometric immersion, its shape operator $S$ has the matrix representation
\begin{equation}\label{eq4:prop:warpedproduct1:ex:Warped}
S\sim\left(\begin{array}{cccc}
\kappa  & 0 & 0 & 0 \\
 0 & \frac{\phi ' +\lambda _0 \varphi ' }{\lambda _0 \phi  -\varphi  } & \frac{\phi   \left(\phi   \phi ' +\varphi   \varphi ' \right)}{\left(\varphi  -\lambda _0 \phi  \right){}^3} & -\frac{\phi   \phi ' +\varphi   \varphi ' }{\left(\varphi  -\lambda _0 \phi  \right){}^2} \\
 0 & 0 & \frac{\phi ' +\lambda _0 \varphi ' }{\lambda _0 \phi  -\varphi  } & 0 \\
 0 & 0 & \frac{\phi   \phi ' +\varphi   \varphi ' }{\left(\varphi  -\lambda _0 \phi  \right){}^2} & \frac{\phi ' +\lambda _0 \varphi ' }{\lambda _0 \phi  -\varphi  } \\
\end{array}
\right)
\end{equation}
 with respect to $\{\partial_x,E_1^h,E_2^h,E_3^h\}$. By combining \eqref{eq3:prop:warpedproduct1:ex:Warped} and \eqref{eq4:prop:warpedproduct1:ex:Warped}, we get (iii) and (iv) of the proposition.
\end{proof}

Next, we define two distributions $D=\mathrm{span}\{e_1,e_4\}$ and $D^\perp=\mathrm{span}\{e_2,e_3\}$ and obtain the following proposition.
\begin{proposition}\label{prop:warped}
Let $M$ be a hypersurface satisfying the assumptions given in Remark~\ref{rem:assumptions}. Then the following conditions are equivalent:
\begin{enumerate}
\item[(i)] $M$ is locally a warped product $I\times_\varphi N^3_1$ with $\partial_x=e_1$,
\item[(ii)] the distribution $D$ is involutive and $\omega_{13}(e_3)=0$,
\item[(iii)] $\omega_{13}(e_3)=\omega_{13}(e_4)=\omega_{34}(e_1)=0$.
\end{enumerate}
In this case, we have
\begin{equation}\label{eq:prop:warped}
\frac{\varphi'}{\varphi}=\omega_{23}(e_1)=-\frac{\lambda'}{\lambda-\mu}.
\end{equation}
\end{proposition}

\begin{proof}
(i)$\Rightarrow$(iii). If $M=I\times_\varphi N^3_1$ locally with $\partial_x=e_1$, then $\nabla_X e_1=\frac{\varphi'}{\varphi}X$ whenever $X$ is tangent to $N^3_1$. By comparing this equation for $X=e_2,e_3,e_4$ with \eqref{eq:nabla11}, we get $\omega_{13}(e_3)=\omega_{14}(e_3)=\omega_{13}(e_4)=0$ and $\omega_{14}(e_4)=\varphi'/\varphi$. Then, Lemma~\ref{lem:ConnForms} gives $\omega_{34}(e_1)=0$.

(iii)$\Rightarrow$(ii). Because of Lemma~\ref{lem:ConnForms}, (iii) implies $\omega_{24}(e_1)=0$. Therefore, \eqref{eq:cor:lie-brackets} gives $[e_1,e_4]=-\omega_{23}(e_1) e_4\in D$.

(ii)$\Rightarrow$(i). Since $D$ is involutive, \eqref{eq:cor:lie-brackets} implies $\omega_{13}(e_4)+\omega_{34}(e_1)=0$ and $\omega_{24}(e_1)=0$. By combining these equations with $\omega_{13}(e_3)=0$ and \eqref{eq:wrel4}, we obtain (iii). Consequently, Lemma~\ref{lem:ConnForms} gives
\begin{equation*}
\omega_{14}(e_3)=\omega_{24}(e_1)=0,\qquad \omega_{14}(e_4)=-\omega_{12}(e_3)=-\omega_{13}(e_2)=\omega_{23}(e_1),
\end{equation*}
and \eqref{eq:cor:lie-brackets} shows that none of the Lie brackets $[e_i,e_j]$ has an $e_1$-component. Therefore, the 1-form dual to $e_1$ is closed, i.e., $e_1=\nabla x$ locally for a function $x$. Moreover, by \eqref{eq:muconst} we have
$$
e_a\big(e_1(\lambda)\big)=[e_a,e_1](\lambda)=0,\qquad a=2,3,4.
$$
Since $e_1(\lambda)=-(\lambda-\mu) \omega_{23}(e_1)$ by \eqref{eq:e1lambda} and $e_a(\lambda-\mu)=0$, we get $e_a\big(\omega_{23}(e_1)\big)=0$. Hence, $\theta:=\omega_{23}(e_1)$ depends only on $x$, and
$$
\nabla_X e_1=\theta(x)\big(X-\langle X,e_1\rangle e_1\big)
$$
for any $X$ tangent to $M$. Hence, $e_1$ spans a totally geodesic foliation, and the level hypersurfaces of $x$ are totally umbilical with the parallel mean curvature vector $-\theta e_1$. Therefore, $M$ is locally a warped product $I\times_\varphi N^3_1$ with $\varphi'/\varphi=\theta$  by Hiepko's theorem \cite{Hiepko1979} (see also \cite{PongeReckziegel1993}). Finally, the second equality in \eqref{eq:prop:warped} follows from \eqref{eq:e1lambda}, because $\lambda=\lambda(x)$. 
\end{proof}

\begin{remark}
The assumption (A4) is not needed in the remaining part of this section because if $M$ is locally a warped product $I\times_\varphi N^3_1$ with $e_1=\partial_x$, then the leaves $I\times\{q\}$,
$q\in N^3_1$, are totally geodesic in $M$. Hence the integral curves of $e_1$ are geodesics.
\end{remark}

When $M$ is locally a warped product $I\times_\varphi N^3_1$ as in Proposition~\ref{prop:warped}, \eqref{eq:cor:lie-brackets} reduces to
\begin{subequations}\label{eq02:Warped}
\begin{align}
\label{eq02a:Warped} [e_1,e_2] &= -2 \frac{\varphi'}{\varphi} e_2, \\
\label{eq02b:Warped} [e_1,e_3] &= 0, \\
\label{eq02c:Warped} [e_1,e_4] &= -\frac{\varphi'}{\varphi} e_4, \\
\label{eq02d:Warped} [e_2,e_3] &= 2 \omega_{34}(e_4) e_2-\omega_{34}(e_2) e_4, \\
\label{eq02e:Warped} [e_2,e_4] &= 5 \omega_{34}(e_2) e_2, \\
\label{eq02f:Warped} [e_3,e_4] &= \omega_{34}(e_3) e_2-2 \omega_{34}(e_2) e_3-\omega_{34}(e_4) e_4.
\end{align}
\end{subequations}
Moreover, because of Lemma~\ref{lem:ConnForms} and \eqref{eq:nabla11}, the Levi-Civita connection of $M$ satisfies
\begin{subequations}\label{eq01:Warped}
\begin{align}
\nabla_{e_1}e_1 &= 0, \quad \nabla_{e_1}e_2 = -\frac{\varphi'}{\varphi} e_2, \quad \nabla_{e_1}e_3 = \frac{\varphi'}{\varphi} e_3, \quad \nabla_{e_1}e_4 = 0, \\
\nabla_{e_2}e_1 &= \frac{\varphi'}{\varphi} e_2, \quad \nabla_{e_2}e_2 = 0, \quad \nabla_{e_2}e_3 = \frac{\varphi'}{\varphi} e_1+\omega_{34}(e_2) e_4, \quad \nabla_{e_2}e_4 = \omega_{34}(e_2) e_2, \\
\nabla_{e_3}e_1 &= \frac{\varphi'}{\varphi} e_3, \quad \nabla_{e_3}e_2 = \frac{\varphi'}{\varphi} e_1-2 \omega_{34}(e_4) e_2+2 \omega_{34}(e_2) e_4, \\
\nabla_{e_3}e_3 &= 2 \omega_{34}(e_4) e_3+\omega_{34}(e_3) e_4, \quad \nabla_{e_3}e_4 = \omega_{34}(e_3) e_2+2 \omega_{34}(e_2) e_3, \\
\nabla_{e_4}e_1 &= \frac{\varphi'}{\varphi} e_4, \quad \nabla_{e_4}e_2 = -4 \omega_{34}(e_2) e_2, \\
\nabla_{e_4}e_3 &= 4 \omega_{34}(e_2) e_3+\omega_{34}(e_4) e_4, \quad \nabla_{e_4}e_4 = -\frac{\varphi'}{\varphi} e_1+\omega_{34}(e_4) e_2.
\end{align}
\end{subequations}

In the following proposition, we obtain an adapted coordinate system for the pseudo-orthonormal frame field $\{e_1,e_2,e_3,e_4\}$ in this case.

\begin{proposition}\label{prop:adapted-frame}
Let $M$ be a hypersurface satisfying the assumptions given in Remark~\ref{rem:assumptions} and $p\in M$. If $M$ is locally a warped product $I\times_\varphi N^3_1$ as in Proposition~\ref{prop:warped}, then there exists a local coordinate system $(x,u,y,v)$ defined on a neighborhood $\mathcal N_p$ of $p$ in which
\begin{equation}\label{eq:prop:adapted-frame}
\begin{split}
e_1 &= \partial_x, \qquad\qquad\qquad\qquad\quad\ \ e_2 = \varphi(x)^{-2}\rho_1(u,y,v) \partial_u, \\
e_3 &= \rho_2 \partial_u+\rho_3 \partial_y+\rho_4(y,v) \partial_v, \qquad e_4 = \varphi(x)^{-1}\rho_5(y,v) \partial_y,
\end{split}
\end{equation}
for some smooth functions $\rho_1,\rho_2,\rho_3$ depending only on $(u,y,v)$ and $\rho_4,\rho_5$ depending only on $(y,v)$, where $\rho_1$, $\rho_4$ and $\rho_5$ are nowhere zero.
\end{proposition}

\begin{proof}
By Proposition~\ref{prop:warped}, $\mathcal N_p$ can be chosen so that $\mathcal N_p=I\times N$ and $e_1=\partial_x$, where $x$ is the coordinate on $I$. Since $e_1(x)=1$ and $e_a(x)=0$, the vector fields $e_2,e_3,e_4$ are tangent to the slices $\{x\}\times N$. Moreover, \eqref{eq02a:Warped}, \eqref{eq02b:Warped} and \eqref{eq02c:Warped} imply that
$$E_2:=\varphi^2 e_2, \qquad E_3:=e_3, \qquad E_4:=\varphi  e_4$$
commute with $\partial_x$. Therefore, $E_2,E_3,E_4$ are the lifts of vector fields on $N$, which we denote by the same letters.

Because of \eqref{eq02e:Warped}, the distribution $\mathrm{span}\{E_2,E_4\}$ on $N$ is involutive. Therefore, there exists a function $v$ on $N$ with $E_2(v)=E_4(v)=0$ and $E_3(v)\neq0$. Furthermore, there exist functions $u$ and $y$ on $N$ satisfying $E_4(u)=0$, $E_2(u)\neq0$ and $E_2(y)=0$, $E_4(y)\neq0$. Then $(u,y,v)$ is a local coordinate system on $N$, and
$$E_2=\rho_1 \partial_u,\qquad E_3=\rho_2 \partial_u+\rho_3 \partial_y+\rho_4 \partial_v,\qquad E_4=\rho_5 \partial_y$$
for some smooth functions $\rho_1,\dots,\rho_5$ of $(u,y,v)$ with $\rho_1\rho_4\rho_5\neq0$. The $\partial_y$-component of \eqref{eq02e:Warped} gives $\partial_u\rho_5=0$, and the $\partial_v$-component of \eqref{eq02d:Warped} gives $\partial_u\rho_4=0$. Hence, we obtain \eqref{eq:prop:adapted-frame}.
\end{proof}

Next, we describe the fibers $\{x_0\}\times N^3_1$ of a warped product hypersurface $I\times_\varphi N^3_1$ satisfying the assumptions given in Remark~\ref{rem:assumptions}.

\begin{proposition}\label{prop:leaves}
Let $M$ be a hypersurface satisfying the assumptions given in Remark~\ref{rem:assumptions}. If $M$ is locally a warped product $I\times_\varphi N^3_1$ as in Proposition~\ref{prop:warped}, then for any $x_0\in I$, there exist $c\in\mathbb E^5_1,\ r>0$ and $\lambda_0\in\mathbb R$ such that the slice $\{x_0\}\times N^3_1$ lies in the de Sitter space
$$
\mathbb S^4_1(1/r^2,c)=\{q\in\mathbb E^5_1:\ \langle q-c,q-c\rangle=r^2\}
$$
and is an isoparametric hypersurface of $\mathbb S^4_1(1/r^2,c)$ whose shape operator
$\hat S$ takes the form
\begin{equation}\label{eq:shape:hatM}
\hat S \sim
\begin{pmatrix}
\lambda_0 & 0 & 1 \\
0 & \lambda_0 & 0 \\
0 & -1 & \lambda_0
\end{pmatrix}
\end{equation}
with respect to a suitable pseudo-orthonormal frame.
\end{proposition}

\begin{proof}
Let $\psi(x,y,u,v)$ be the position vector of  $M$, $\xi$ be its unit normal vector field and $x_0\in I$. We put $\hat\psi(y,u,v)=\psi(x_0,y,u,v)$ for the position vector of the slice $\hat M=\{x_0\}\times N^3_1$ of $M$ and consider the frame field $\{\hat e_1,\hat e_2,\hat e_3;\hat e_4,\hat e_5\}$ on $\hat M$ defined by
$$\hat e_i=e_{i+1}\left.\right|_{\hat M},\ i=1,2,3,\qquad \hat e_4=e_{1}\left.\right|_{\hat M},\hat e_5=\xi\left.\right|_{\hat M}, $$
where $\hat e_4,\hat e_5$ are normal to $\hat M$. Then, \eqref{eq01:Warped} yields that $\hat e_4$ is parallel in the normal bundle of $\hat M$ and 
$$A_{\hat e_4} \hat e_i=-\frac 1r \hat e_i,$$
where we put $\frac 1r=\frac{\varphi'(x_0)}{\varphi(x_0)}$. Therefore, we have 
$$\widetilde\nabla_{\hat e_i}\left( \hat\psi-r\hat e_4\right)=0,$$
which yields that 
$$\hat\psi=c+r\hat e_4$$
for a constant vector $c\in\mathbb E^5_1$, i.e., $\hat M$ lays on  $\mathbb S^4_1(1/r^2,c)$. Moreover, $\hat e_5$ is the unit normal of $\hat M$ in  $\mathbb S^4_1(1/r^2,c)$. Hence, the shape operator $\hat S$ of $\hat M$ has the matrix representation given by \eqref{eq:shape:hatM} for $\lambda_0=\lambda (x_0).$
\end{proof}

Now, we obtain further geometric properties of a warped product hypersurface $I\times_\varphi N^3_1$ satisfying the assumptions given in Remark~\ref{rem:assumptions}.

\begin{proposition}\label{prop:warpedproduct1}
Let $M$ be a hypersurface satisfying the assumptions given in Remark~\ref{rem:assumptions}.
If $M$ is locally a warped product $I\times_\varphi N^3_1$ as in
Proposition~\ref{prop:warped}, then
\begin{enumerate}
\item[(i)] $\mu=0$,
\item[(ii)] $N^3_1$ has constant scalar curvature,
\item[(iii)] up to a translation of $x$, the warping function is $\varphi(x)=ax$ for a nonzero constant $a$.
\end{enumerate}
\end{proposition}

\begin{proof}
Let $\{e_1,e_2,e_3,e_4\}$ be as in Proposition~\ref{prop:adapted-frame} with \eqref{eq:shape0}. A direct computation shows that the metric of $M$ in the coordinates $(x,u,y,v)$ is
\begin{equation}\label{eq:metric:adapted}
g= dx^2+\varphi(x)^2 \hat g,\qquad
\hat g=-\frac{2}{\rho_1\rho_4} du dv
+\frac{2\rho_2}{\rho_1\rho_4^2} dv^2
+\frac{1}{\rho_5^2}\Big(dy-\frac{\rho_3}{\rho_4} dv\Big)^2,
\end{equation}
where $\hat g$ is the metric of $N$. Moreover, $e_1$ is an eigenvector of the Ricci operator of $M$ with
\begin{equation}\label{eq:Ric:e1}
\operatorname{Ric}(e_1)=-\frac{3\varphi''}{\varphi} e_1,
\end{equation}
and the scalar curvature of $M$ is
\begin{equation}\label{eq:scalar:adapted}
\operatorname{Scal}=\frac{\widehat{\operatorname{Scal}}}{\varphi^2}-\frac{6\big(\varphi\varphi''+\varphi'^2\big)}{\varphi^2},
\end{equation}
where $\widehat{\operatorname{Scal}}$ denotes the scalar curvature of $(N^3_1,\hat g)$. By combining \eqref{eq:scalar:adapted} with (iv) of Lemma \ref{lem:RicciOperator}, we obtain that the scalar curvature of  $N$ is constant. Thus, we have (ii) of the proposition. 

On the other hand, by combining \eqref{eq:gauss}  for $X=e_3, Y=Z=e_1,W=e_3$ with \eqref{eq:second-fundamental-form}, \eqref{eq01:Warped} and \eqref{eq02c:Warped}, we get (i) of the proposition. So, (iii) of Lemma \ref{lem:RicciOperator} implies $\operatorname{Ric}(e_1)=0$. Therefore, \eqref{eq:Ric:e1} implies $\varphi''=0$. Therefore, we have (iii) of the proposition. 
\end{proof}

Now, we are to prove the main result of this section.
\begin{theorem}\label{thm:warped-classification}
Let $M$ be a hypersurface of $\mathbb E^5_1$ satisfying the assumptions given in Remark~\ref{rem:assumptions}, except possibly (A4). $M$ is locally a warped product $I\times_\varphi N^3_1$  with $e_1$ is tangent to $I$, if and only if it is locally congruent to the hypersurface constructed in Example~\ref{ex:Warped} for a $\lambda_0\in\mathbb R$ and $\alpha(x)=(0,x),\ x\neq 0$.
\end{theorem}

\begin{proof}
In order to prove the necessary condition, we assume that $M$ is locally a warped product $I\times_\varphi N^3_1$. Let $\psi$ be the position vector of $M$. By combining \eqref{eq01:Warped} and \eqref{eq:second-fundamental-form} with Proposition~\ref{prop:adapted-frame} and Proposition~\ref{prop:warpedproduct1}, we obtain
\begin{equation*}
\psi_{xx}=0,\qquad \psi_{xy}=\frac{1}{x}\psi_{y},\qquad \psi_{xu}=\frac{1}{x}\psi_{u},\qquad \psi_{xv}=\frac{1}{x} \psi_{v}.
\end{equation*}
which yields that, up to a suitable translation,
\begin{equation}\label{eq1:thm:warped-classification}
\psi(x,y,u,v)=x\psi_1(y,u,v)
\end{equation}
for a $\bar\psi$. Since the slice  $\psi(x_0,y,u,v)=x_0\bar\psi(y,u,v)$ lays on $\mathbb S^4_1(1/r^2,c)$ for some $r>0$ and $c\in\mathbb E^5_1$ because of Proposition~\ref{prop:leaves},  by re-defining $x$ and $c$ properly, we observe that \eqref{eq1:thm:warped-classification} turns into
\begin{equation}\label{eq2:thm:warped-classification}
\psi(x,y,u,v)=xc+x\hat\psi(y,u,v), \qquad \langle\hat\psi,\hat\psi\rangle=1.
\end{equation}
Note that because of \eqref{eq:prop:adapted-frame} in Proposition~\ref{prop:adapted-frame}, we have $\langle\partial_x,\partial_y\rangle=\langle\partial_x,\partial_u\rangle=\langle\partial_x,\partial_v\rangle=0$. Therefore, \eqref{eq2:thm:warped-classification} implies $c=0$. Moreover, Proposition~\ref{prop:leaves} yields that $\bar\psi$ is the position vector of an isoparametric hypersurface of $\mathbb S^4_1(1/r^2,0)$ whose shape operator $\hat S$ takes the form \eqref{eq:shape:hatM} for a $\lambda_0$. Hence, $M$ is locally congruent to the hypersurface constructed in Example~\ref{ex:Warped} for  $\alpha(x)=(0,x)$. (i) and (iii) of Proposition~\ref{prop:warpedproduct1:ex:Warped} yields  $x\neq0$.

The converse follows from  Proposition~\ref{prop:warpedproduct1:ex:Warped}.
\end{proof}

%%%%%%%%%%%%%%%%%%%%%%
%%%%%%%%%%%%%%%%%%%%%%
%%%%%%%%%%%%%%%%%%%%%%
%%%%%%%%%%%%%%%%%%%%%%
%%%%%%%%%%%%%%%%%%%%%%

\section*{Acknowledgements}
During the preparation of this work the authors used Claude (Anthropic) in order to improve the language and readability of the manuscript. \\
Zülal Derin Yaqub gratefully acknowledges the support provided by the Scientific and Technological Research Council of Türkiye, TÜBİTAK-2218 – National Postdoctoral Research Fellowship Program, under project number 123C433.

\section*{Declarations}

\begin{itemize}
\item \textbf{Funding:} Z\"ulal Derin Yaqub is supported by T\"UB\.ITAK under the 2218 -- National Postdoctoral Research Fellowship Program (project number 123C433), hosted by Istanbul Technical University.
\item \textbf{Conflict of interest:} The authors declare no competing interests.
\item \textbf{Ethics approval and consent to participate:} Not applicable.
\item \textbf{Consent for publication:} Not applicable.
\item \textbf{Data availability:} Data sharing is not applicable to this article as no datasets were generated or analysed during the current study.
\item \textbf{Materials availability:} Not applicable.
\item \textbf{Code availability:} Not applicable.
\item \textbf{Author contribution:} All authors contributed to the study conception, the analysis and the writing of the manuscript. All authors read and approved the final manuscript.
\end{itemize}
\noindent

\end{document}